\documentclass[11pt,a4paper,reqno]{amsart}%
\usepackage{amssymb}
\usepackage{amsmath}
\usepackage{amsfonts}
\usepackage[T1]{fontenc}
\usepackage{color}
\usepackage[usenames,dvipsnames]{xcolor}
\usepackage{bbm}
\usepackage{graphicx}
\usepackage{rotating}
\usepackage{hyperref}
\usepackage{mathrsfs}
\usepackage[all]{xy}%
\providecommand{\U}[1]{\protect\rule{.1in}{.1in}}
\newtheorem{theorem}{Theorem}

\newtheorem{definition}[theorem]{Definition}
\newtheorem{example}[theorem]{Example}

\newtheorem{lemma}[theorem]{Lemma}

\newtheorem{proposition}[theorem]{Proposition}
\newtheorem{remark}[theorem]{Remark}

\thanks{}
\email{lvitagliano@unisa.it}
\begin{document}
\title[$L_{\infty}$-algebras from multicontact geometry]{$L_{\infty}$-algebras from multicontact geometry}
\author{Luca Vitagliano}
\address{DipMat, Universit\`a degli Studi di Salerno, {\& Istituto Nazionale di Fisica
Nucleare, GC Salerno,} Via Ponte don Melillo, 84084 Fisciano (SA), Italy.}

\begin{abstract}
I define higher {codimensional} versions of contact structures on manifolds as maximally
non-integrable distributions. I call them multicontact structures. Cartan
distributions on jet spaces provide canonical examples. More generally, I
define higher {codimensional} versions of pre-contact structures as distributions on manifolds
whose characteristic symmetries span a constant dimensional distribution. {I call them pre-multicontact structures.}
Every distribution is almost everywhere, locally, a pre-multicontact
structure. After showing that the standard symplectization of contact
manifolds generalizes naturally to a (pre-)multisymplectization of
(pre-)multicontact manifolds, I make use of results by C. Rogers and M. Zambon
to associate a canonical $L_{\infty}$-algebra to any (pre-)multicontact
structure. Such $L_{\infty}$-algebra is a {multicontact} version of the Jacobi
bracket on a contact manifold. {However, unlike the \emph{multisymplectic $L_\infty$-algebra} of Rogers and Zambon, the \emph{multicontact $L_\infty$-algebra} is always a homological resolution of a Lie algebra.} Finally, I
describe in local coordinates the $L_{\infty}$-algebra associated to the
Cartan distribution on jet spaces.
\end{abstract}
\maketitle



\section{Introduction}

A \emph{contact manifold} is a smooth manifold equipped with a \emph{contact
structure}, i.e.~a maximally non-integrable hyperplane distribution. A
canonical example of a contact manifold is provided by the space of first jets
of hypersurfaces in a given manifold. Accordingly, contact geometry, i.e.~the
theory of contact structures, is at the foundation of the theory of first
order partial differential equations in one dependent variable (see, for
instance, \cite{b...99}). Every contact manifold is naturally equipped with a
\emph{Jacobi bundle}, i.e.~a line bundle with a Lie bracket on sections,
which is a first order differential operator in each entry. The Lie algebra of
sections of the Jacobi bundle of a contact manifold is canonically isomorphic
to the Lie algebra of \emph{contact vector fields}, i.e.~infinitesimal
symmetries of the contact structure \cite{cs13,m91}.

On another hand, contact manifolds can be understood as odd dimensional
analogues of symplectic manifolds and there is a close relationship between
contact and symplectic geometry. In particular, every contact manifold can be
`` extended'' in a natural way to a symplectic
manifold, its \emph{symplectization}, encoding all the information about the
contact structure. For instance, the Poisson algebra of the symplectization
`` knows everything'' about the Jacobi bundle
\cite{m91}.

There are higher {degree} versions of symplectic manifolds, namely
\emph{multisymplectic manifolds}. They are smooth manifolds equipped with a
multisymplectic structure, i.e.~a higher degree, closed, non-degenerate
differential form\emph{ }(see, for instance, \cite{cdi99}). A multisymplectic
manifold is sometimes called $n$\emph{-plectic} if its multisymplectic
structure is of degree $n$. Thus, $1$-plectic manifolds are standard
symplectic manifolds. In a similar way as symplectic geometry is at the
foundation of classical mechanics, multisymplectic geometry is at the
foundation of classical field theory. There is a {multisymplectic} analogue of the
Poisson algebra of a symplectic manifold. Namely, every multisymplectic
structure gives rise to an $L_{\infty}$-algebra \cite{r12} (see also
\cite{z12}), i.e.~a cochain complex with a bracket satisfying the Jacobi
identity only up to (a coherent system of higher) homotopies \cite{ls93,lm95}.
In the same way as elements in the Poisson algebra of a symplectic manifold
are interpreted as observables in mechanics, elements in the $L_{\infty}%
$-algebra of a multisymplectic manifold should be interpreted as observables
in field theory \cite{bhr19,frs13}.

In this paper, I introduce higher {codimensional} versions of contact manifolds. I call them
\emph{multicontact manifolds}. They are smooth manifolds equipped with a
\emph{multicontact structure}, i.e.~a maximally non-integrable distribution
of higher codimension. I will call $n$\emph{-contact }a multicontact manifold
whose multicontact structure is $n$-codimensional. Thus, $1$-contact manifolds
are standard contact manifolds. Higher order jet spaces are canonical examples
of multicontact manifolds. Interestingly, there is a nice relationship between
multicontact geometry and multisymplectic geometry: every $n$-contact manifold
can be `` extended'' in a natural way to an
$n$-plectic manifold, its \emph{multisymplectization}, encoding all the
information about the $n$-contact structure. Moreover, there is a {multicontact}
analogue of the Jacobi bundle of a contact manifold. Namely, every
multicontact structure gives rise to an $L_{\infty}$-algebra. The latter, is
in the same relation with the $L_{\infty}$-algebra of the multisymplectization
as the Jacobi bundle of a contact manifold is with the Poisson algebra of the symplectization. 

Finally, recall that, relaxing the non-degeneracy condition in the definition
of a symplect form, one gets the (much more general) notion of
\emph{pre-symplectic form}. A pre-symplectic form is just a closed
differential $2$-form. Similarly, relaxing the non-degeneracy condition in the
definition of a multisymplectic form, one gets the (much more general) notion
of \emph{pre-multisymplectic form}. A pre-multisymplectic form is just a
closed differential form. Pre-multisymplectic forms give rise to $L_{\infty}%
$-algebras of observables as well \cite{z12}. I will show that similar
considerations hold in the `` contact realm'' .
Namely, relaxing the maximality condition in the definition of a contact
distribution, one gets the (much more general) notion of \emph{pre-contact
distribution}. A pre-contact distribution is just an hyperplane distribution.
Similarly, relaxing the maximality condition in the definition of multicontact
distribution, one gets the (much more general) notion of pre-multicontact
distribution. A pre-multicontact distribution is just a distribution
(fulfilling a conceptually unessential, additional, regularity property). As
such, it is a very general notion. Indeed, distributions are ubiquitous in
differential geometry. The main reason is that any partial differential
equation can be understood geometrically as a manifold with a, generically
non-integrable, distribution. Solutions then identify with integral
submanifolds of a suitable dimension. Below, I show that pre-multicontact
distributions give rise to $L_{\infty}$-algebras as well. In an appendix I
also provide coordinate formulas for the higher brackets in the $L_{\infty}%
$-algebras of higher order jet spaces.

{Notice that there is a substantial difference between the multisymplectic and the multicontact cases. Namely, the homology of the \emph{multisymplectic $L_\infty$-algebra} projects onto Hamiltonian vector fields (Definition \ref{def:lHvf}), the kernel of the projection being $(n-1)$-th de Rham cohomologies, and possesses generically non-trivial contributions in positive degrees. As a consequence, it is not \emph{formal}, in general, i.e.~it is not quasi-isomorphic to its homology equipped with the induced Lie bracket. Actually, by homotopy transfer, it induces in homology a new $L_\infty$-algebra structure with generically non-trivial higher operations. In particular, from an homotopical algebra point of view, \emph{the multisymplectic $L_\infty$-algebra contains more information than its cohomology}. In particular, it contains more information than Hamiltonian vector fields (see \cite{v12} for a smooth introduction to the homotopy theory of $L_\infty$-algebras). On the other hand, the \emph{multicontact $L_\infty$-algebra} is always a resolution of the Lie algebra of infinitesimal symmetries of the multicontact structure. In particular, it is formal and, from a homotopical algebra point of view, contains the same information as infinitesimal symmetries.}

\subsection{Notations and conventions}

Let $M$ be a smooth manifold. I denote by $C^{\infty}(M)$ the algebra of
real-valued smooth functions on $M$. Moreover, I denote by $\mathfrak{X}(M)$
vector fields on $M$. I always understand vector fields as derivations of the
algebra $C^{\infty}(M)$. I denote by $\Omega(M)=\bigoplus_{k}\Omega^{k}(M)$
differential forms on $M$ and by $d:\Omega(M)\to\Omega(M)$ the
exterior differential. I denote by $i_{X}$, and $L_{X}$ the ``insertion of a
vector field $X$ into'' and the ``Lie derivative along $X$ of'' differential forms
respectively. If $V\to M$ is a vector bundle over $M$, I denote by
$V^{\ast}\to M$ the dual bundle. If $\upsilon$ is a section of
$V$, I denote by $\upsilon_{x}$ its value at $x\in M$. Finally, I adopt the
Einstein summation convention on pairs of upper-lower indexes.

\section{Distributions, Contact Manifolds and Symplectization\label{SecCont}}

In this section I collect my notations, and basic facts, about distributions
on manifolds. Let $M$ be a smooth manifold and $C$ a regular distribution on it, i.e.,
a linear subbundle of the tangent bundle $TM$ of $M$. I denote by
$C_{x}\subset T_{x}M$ the fiber of $C$ growing over $x\in M$. The \emph{rank}, or \emph{dimension}, of $C$ is $\operatorname{rank} C := \dim C_x$, where $x$ is any point in $M$. The
\emph{annihilator of} $C$ is the linear subbundle $C^{0}$ of the cotangent
bundle $T^{\ast}M$ consisting of $1$-forms vanishing on vectors in $C$. I
denote by $\mathfrak{X}_{C}$ the Lie algebra of \emph{infinitesimal symmetries
of }$C$, i.e.~those vector fields on $M$ whose flow preserves $C$. The Lie
algebra $\mathfrak{X}_{C}$ is also the stabilizer of the subspace $\Gamma(C)$
in the Lie algebra $\mathfrak{X}(M)$ of vector fields on $M$. Denote by
$N:=TM/C$ the normal bundle to $C$. Thus, there is a natural $C^{\infty}%
(M)$-linear projection $\theta:\mathfrak{X}(M)\to\Gamma(N)$, a
short exact sequence
\[
0\longrightarrow\Gamma(C)\longrightarrow\mathfrak{X}(M)\overset{\theta
}{\longrightarrow}\Gamma(N)\longrightarrow0,
\]
and a dual exact sequence
\[
0\longleftarrow\Gamma(C^{\ast})\longleftarrow\Omega^{1}(M)\longleftarrow
\Gamma(C^{0})\longleftarrow0,
\]
where I identified $C^{0}$ with the dual bundle $N^{\ast}$ of $N$. The
\emph{curvature} of $C$ is the well defined skew-symmetric $C^{\infty}%
(M)$-bilinear map
\[
R:\Gamma(C)\times\Gamma(C)\longrightarrow\Gamma(N),\quad(X,Y)\longmapsto
\theta([X,Y]).
\]
Clearly, $C$ is integrable, i.e.~$\Gamma(C)$ is a Lie subalgebra in
$\mathfrak{X}(M)$, iff $R=0$.

The \emph{characteristic distribution} $D$ of $C$ consists of tangent vectors
$\zeta$ in $C$ such that $R(\zeta,-)=0$. In general, $D$ is not regular, i.e.~its rank may change along $M$. Notice that $\Gamma(D)\subset
\mathfrak{X}_{C}$. Accordingly, sections of $D$ are also called
\emph{characteristic symmetries} of $C$. A distribution $C$ such that $D = 0$ is called \emph{maximally non-integrable}. It is easy to see that
$D$ is integrable whenever its rank is constant.

\begin{remark}
[coordinate formulas]\label{19}Given a distribution $C$ on a manifold $M$, I
will always choose coordinates {$(x^{i},z^{a})$} on $M$ which
are adapted to $C$, i.e.~such that $\Gamma(C)$ is locally spanned by vector
fields $C_{i}:=\frac{\partial}{\partial x^{i}}+C_{i}^{a}\frac{\partial
}{\partial z^a}$, and $\Gamma(C^{0})$ is locally spanned by differential
$1$-forms $\vartheta^{a}:=dz^{a}-C_{i}^{a}dx^{i}$, {where $i = 1, \ldots, \operatorname{rank} C$, and $a = 1, \ldots,  \dim M - \operatorname{rank} C$. In the following, indexes $i,j$ will always run in the range $ 1, \ldots, \operatorname{rank} C$, while indexes $a,b$ will run in the range $1, \ldots, \dim M - \operatorname{rank} C$}. Put $\partial_{a}%
:=\partial/\partial z^{a}$. It is easy to see that
\[
\left[  \partial_{a},C_{i}\right]  =\partial_{a}C_{i}^{b}\partial_{b}%
,\quad\text{and\quad}[C_{i},C_{j}]=R_{ij}^{a}\partial_{a},\quad R_{ij}%
^{a}:=C_{i}(C_{j}^{a})-C_{j}(C_{i}^{a}).
\]
Dually,%
\[
d\vartheta^{a}=\partial_{b}C_{i}^{a}dx^{i}\wedge\vartheta^{b}-\dfrac{1}%
{2}R_{ij}^{a}dx^{i}\wedge dx^{j}.
\]
Sections $\ldots,\theta(\partial_{a}),\ldots$ of the normal bundle $N$ (resp.,
sections $dx^{i}|_{C}$ of the dual bundle of $C$) form a local basis.
Locally,
\[
\theta=\vartheta^{a}\otimes\theta(\partial_{a})
\]
and
\[
R=\dfrac{1}{2}R_{ij}^{a}\,dx^{i}|_{C}\wedge dx^{j}|_{C}\otimes\theta
(\partial_{a}).
\]
A tangen vector $\zeta=\zeta^{i}C_{i}$ in $C$ belongs to the characteristic
distribution $D$ iff $R_{ij}^{a}\zeta^{j}=0$.
\end{remark}

Now, recall that a \emph{contact structure}, or a \emph{contact distribution},
is a maximally non-integrable, hyperplane distribution. A \emph{contact
manifold} is a manifold $M$ equipped with a contact distribution $C$. The
normal bundle $N=TM/C$ of a contact distribution $C$ is naturally a
\emph{Jacobi bundle}, or \emph{Jacobi structure}, on $M$, i.e. a line bundle
with a Lie bracket on $\Gamma(N)$ which is a differential operator of order
$1$ in each entry. Indeed, the map $\mathfrak{X}_{C}\to\Gamma(N)$,
$X\mapsto\theta(Y)$, is a vector space isomorphism. In particular,
$\Gamma(N)$ inherits from $\mathfrak{X}_{C}$ a Lie bracket, the \emph{Jacobi
bracket}, with the required bi-differential operator property.

There is a natural way of `` producing'' a
symplectic manifold $(\widetilde{M},\widetilde{\omega})$ from a contact
manifold $(M,C)$, called the \emph{symplectization}. Basically, $(\widetilde
{M},\widetilde{\omega})$ contains a full information about $(M,C)$. Let me
recall the construction of $(\widetilde{M},\widetilde{\omega})$. First of all,
one defines $\widetilde{M}$ as $C^{0}$ with the image of the zero section
removed. In particular, the projection $\pi:\widetilde{M}\to M$ is
a principal bundle with structure group $\mathbb{R}^{\times}$. Now, notice
that $\widetilde{M}$ is a symplectic submanifold in $T^{\ast}M$, i.e.~the
canonical symplectic structure on $T^{\ast}M$ restricts to a symplectic
structure $\widetilde{\omega}$ on $\widetilde{M}$.

Finally, I describe the relationship between the Jacobi bracket $\{-,-\}$ on
$\Gamma(N)$ and the Poisson bracket $\{-,-\}_{\widetilde{M}}$ on $C^{\infty
}(\widetilde{M})$ (see, for instance, \cite{m91}). First, denote by $\Delta$
the Euler vector field on $C^{0}$. Since $\widetilde{M}$ is open in $C$,
$\Delta$ restricts to it. I denote again by $\Delta$ the restriction. It is
the fundamental vector field corresponding to the canonical generator $1$ in
the Lie algebra $\mathbb{R}$ of the structure group $\mathbb{R}^{\times}$. A
function $f$ on $\widetilde{M}$ is \emph{homogeneous} if $\Delta(f)=f$. Now,
sections of $N$ identify with fiberwise linear functions on $C^{0}$, which,
in their turn, restrict to homogeneous functions on $\widetilde{M}$. Since
$\widetilde{M}$ is dense in $C$, the restriction is injective. Summarizing, a
sections $\nu$ of $N$ identifies with a homogeneous function $\widetilde{\nu}$
on $\widetilde{M}$. Moreover,
\begin{equation}
\widetilde{\{\nu_{1},\nu_{2}\}}=\{\widetilde{\nu}_{1},\widetilde{\nu}%
_{2}\}_{\widetilde{M}}. \label{16}%
\end{equation}
In fact, one could use Formula (\ref{16}) as a definition for the Jacobi
bracket on $\Gamma(N)$.

The main aim of this paper is to provide a ``higher {codimensional}
version'' of the content of this section.

\section{Multicontact Manifolds and Multisymplectization}

In this section, I present my proposal of \emph{higher {codimensional} contact structures}. First recall
the definition of higher (pre)symplectic structure. Let $M$ be a smooth
manifold. If $\mu$ is a differential form on $M$, I denote by $\ker
\mu\subset TM$ the (not necessarily constant rank) distribution
spanned by tangent vectors $\xi$ such that $i_{\xi}\mu=0$. If $d\mu=0$, then
$\ker\mu$ is an integrable distribution provided it has constant rank.

\begin{definition}
[\cite{cdi99,frz13}]A closed differential $(n+1)$-form $\omega$ on $M$ is a
\emph{pre-}$n$\emph{-plectic structure} (or \emph{pre-}$n$\emph{-plectic
form}) if $\ker\omega$ is a constant dimensional (integrable) distribution. A
pre-$n$-plectic structure $\omega$ such that $\ker\omega=0$ is an
$n$\emph{-plectic structure} (or $n$\emph{-plectic form}). A manifold equipped
with a (pre-)$n$-plectic structure is a (\emph{pre-})$n$\emph{-plectic
manifold}.
\end{definition}

Often, (pre-)$n$-plectic structures (resp.~forms, manifolds) are collectively
referred to as (\emph{pre-})\emph{multisymplectic} structures (resp.~forms, manifolds).

\begin{example}
Standard symplectic manifolds are $1$\emph{-plectic manifolds}.
\end{example}

\begin{example}
\label{21}Let $M$ be a manifold. On the bundle $\wedge^{n}T^{\ast}M$ of
$n$-forms over $M$ there is a \emph{tautological }$n$\emph{-form} $\theta_{M}$
defined by
\[
(\theta_{M})_{\alpha}=\mathrm{pr}^{\ast}\alpha,\quad \alpha\in\wedge^{n}T^{\ast}M,
\]
$\mathrm{pr}:\wedge^{n}T^{\ast}M\to M$ being the projection. The
exterior differential $\omega_{M}:=d\theta_{M}$ of $\theta_{M}$ is an
$n$-plectic structure on $\wedge^{n}T^{\ast}M$.
\end{example}

\begin{definition}
An $n$-codimensional distribution $C$ on $M$ is a \emph{pre-}$n$\emph{-contact
structure} (or \emph{pre-}$n$\emph{-contact distribution}) if its
characteristic distribution $D$ is constant dimensional. A pre-$n$-contact
structure such that $D=0$ is an $n$\emph{-contact structure} (or
$n$\emph{-contact distribution}). A manifold equipped with a (pre-)$n$-contact
structure is a (\emph{pre-})$n$\emph{-contact manifold}.
\end{definition}

I will collectively refer to (pre-)$n$-contact structures (resp.~forms,
manifolds) as (\emph{pre-})\emph{multicontact} structures (resp.~forms, manifolds).

\begin{example}
Standard contact manifolds (see Section \ref{SecCont}) are $1$-contact manifolds.
\end{example}

\begin{example}
The \emph{Cartan distribution} on a jet space is a multicontact structure (see
Appendix \ref{SecJet}).
\end{example}

\begin{definition}[\cite{r12,frz13}]\label{def:lHvf}
A vector field $X$ on a pre-multisymplectic manifold
$(M,\omega)$ is \emph{locally Hamiltonian }if it preserves $\omega$, i.e.,
$L_{X}\omega=0$. A differential $(n-1)$-form $\mu$ on $(M,\omega)$ is \emph{Hamiltonian} if there exists a vector field $Y$ such that $i_{Y}%
\omega=-d\mu$. Then, $Y$ is called an \emph{Hamiltonian vector field
}associated to $\mu$.
\end{definition}

Clearly, Hamiltonian vector fields are locally Hamiltonian. Moreover, on a
multisymplectic manifold, every Hamiltonian form possesses a unique associated
Hamiltonian vector field. In the following I define the contact analogues of
(locally) Hamiltonian vector fields (see Section \ref{SecHom} for the contact
analogue of Hamiltonian forms).

\begin{definition}
A vector field $X$ on a pre-multicontact manifold $(M,C)$ is
\emph{multicontact} if its flow preserves $C$, i.e.~$X\in\mathfrak{X}_{C}$.
\end{definition}

Let $(M,C)$ be a pre-multicontact manifold. As in Section \ref{SecCont},
denote by $N=TM/C$ the normal bundle and by $\theta:TM\to N$ the
projection. Clearly, the kernel of the map $\theta:\mathfrak{X}_{C}%
\to\Gamma(N)$ consists of sections of $D$. In
particular, $\Gamma(D)$ is an ideal in the Lie algebra $\mathfrak{X}_{C}$, and
the quotient $\mathfrak{X}_{C}/\Gamma(D)$ is a Lie algebra.

\begin{definition}
Elements in the image of $\theta:\mathfrak{X}_{C}\to\Gamma(N)$ are
called \emph{Hamiltonian sections}. Their collection is denoted by
$\Gamma_{\mathrm{Ham}}(N)$.
\end{definition}

Thus, there is a short exact sequence of Lie algebras
\[
0\longrightarrow \Gamma (D) \longrightarrow\mathfrak{X}_{C}\longrightarrow
\Gamma_{\mathrm{Ham}}(N)\longrightarrow0.
\]
When $C$ is a multicontact structure, i.e.~$D=0$, there is an isomorphism
$\mathfrak{X}_{C}\simeq\Gamma_{\mathrm{Ham}}(N)$, $X\mapsto\theta(X)$.

\begin{remark}
[Coordinate formulas]A vector field $X$ on $M$ locally given by (see Remark
(\ref{19})) $X=X^{a}\partial_{a}+X^{i}C_{i}$ is multicontact iff
\[
C_{i}(X^{a})-\partial_{b}C_{i}^{a}\,X^{b}+R_{ij}^{a}X^{j}=0.
\]
Since $X$ projects onto a section $\theta(X)$ of $N$ which is locally given by
$\theta(X)=X^{a}\theta(\partial_{a})$, one concludes that a section $\nu$ of
$N$ locally given by $\nu=\nu^{a}\theta(\partial_{a})$ is Hamiltonian iff
\begin{equation}
\partial_{b}C_{i}^{a}\,\nu^{b}-C_{i}(\nu^{a})=R_{ij}^{a}X^{j} \label{7}%
\end{equation}
for some local functions $X^{j}$. If $D=0$ then the $X^{j}$'s are uniquely defined.
\end{remark}

There is a canonical way how to associate a (pre-)$n$-plectic manifold to a
(pre-)$n$-contact manifold, generalizing the symplectization procedure
described in Section \ref{SecCont}. Namely, let $(M,C)$ be an pre-$n$-contact
manifold. In the bundle $\wedge^{n}T^{\ast}M$ of $n$-forms, with bundle
projection $\mathrm{pr}:\wedge^{n}T^{\ast}M\to M$, consider the
subset $\widetilde{M}$ consisting of $n$-forms $\alpha$ such that
\[
C_x = \{\xi\in T_{x}M:i_{\xi}\alpha=0\},\quad x=\mathrm{pr}(\alpha).
\]
Since $C$ is $n$-codimensional, $\widetilde{M}$ is a $1$-dimensional subbundle
of $\wedge^{n}T^{\ast}M$. Actually, it coincides with $\wedge^{n}C^{0}$
with the image of the $0$ section removed. In particular, it is a principal
$\mathbb{R}^{\times}$-bundle. Denote by $\pi:$ $\widetilde{M}\to
M$ the projection.

As in Example \ref{21}, denote by $\theta_{M}$ the tautological $n$-form on
$\wedge^{n}T^{\ast}M$, and by $\widetilde{\theta}$ its restriction to
$\widetilde{M}$. Put also $\omega_{M}=d\theta_{M}$, and $\widetilde{\omega
}=d\widetilde{\theta}=\omega|_{\widetilde{M}}$. Finally, recall that
$\omega_{M}$ is an $n$-plectic form. {My next aim is to show that $\widetilde \omega$ is a pre-$n$-plectic form and to describe its kernel. In order to do this some preparatory remarks are needed.}

\begin{remark}
\label{20}Let $A\in\Omega^{n}(M)$ be a (local) section of $\widetilde{M}$. A vector
field $X\in\mathfrak{X}(M)$ is multicontact (in the domain of $A$) iff $L_{X}A=fA$ for a smooth
function $f$ on $M$. Indeed, $L_{X}A$ is proportional to $A$ iff $i_{Y}%
L_{X}A=0$ for all $Y\in\Gamma(C)$. But, since $i_{Y}A=0$, then $i_{Y}%
L_{X}A=i_{[Y,X]}A$ which vanishes iff $[Y,X]\in \Gamma (C)$.
\end{remark}

\begin{remark}
\label{3}Let $A\in\Omega^{n}(M)$ be a section of $\widetilde{M}$. Then
$dA|_{C}=0$ unless $n=1$. In particular, except for the contact case,
$dA|_{C}$ is not multisymplectic. Instead, $dA$ has the following degeneracy
property. Let $\zeta \in C_x$, $x \in M$. Then
\begin{equation} \label{1}
\ker i_{\zeta}dA\supset C_x \Longleftrightarrow \zeta \in D_x.
\end{equation}
Indeed let $Z\in\Gamma(C)$ be such that $Z_x = \zeta$, and let $Y \in \Gamma (C)$. Then
\[
i_{Y_x} i_\zeta dA = i_{Y_x} L_Z A = i_{[Y,Z]_x} A.
\]
In particular $i_{Y_x} i_\zeta dA$ vanishes for all $Y \in \Gamma (C)$ iff $[Y,Z]_x \in C_x$, i.e.
$
0 = \theta ([Y,Z]_x) = - R (\zeta, Y_x). 
$
It follows from the arbitrariness of $Y$ that $\ker i_{\zeta}dA\supset C_x$ iff $R (\zeta , -) = 0$ iff $\zeta \in \Gamma (D)$.
\end{remark}

\begin{remark}
[coordinate formulas]Using the same notations as in Section \ref{SecCont},
Remark \ref{19}, put
\[
\Theta:=\vartheta^{1}\wedge\cdots\wedge\vartheta^{n},\text{\quad and\quad
}\Theta_{a}:=i_{\partial_{a}}\Theta,
\]
and notice that
\[
d\Theta=\partial_{a}C_{i}^{a}dx^{i}\wedge\Theta-\dfrac{1}{2}R_{ij}^{a}%
dx^{i}\wedge dx^{j}\wedge\Theta_{a}.
\]
A section $A$ of $\widetilde{M}$ is locally of the form
\[
A=f\,\Theta,\quad f\in C^{\infty}(M),\quad f\neq0.
\]
Then
\[
dA=\left(  C_{i}(f)+f\,\partial_{a}C_{i}^{a}\right)  dx^{i}\wedge\Theta
-\dfrac{1}{2}R_{ij}^{a}dx^{i}\wedge dx^{j}\wedge\Theta_{a}.
\]

\end{remark}

Now, let $A$ be a local section of $\widetilde M$, and $Z\in\Gamma(D)$. In particular $A$ is a differential forms. Clearly,
$L_{Z}A\in\Gamma(\wedge^{n}C^{0})$. Moreover $L_{Z}A$ is $C^{\infty}
(M)$-linear in $Z$.    Section $A$ is called \emph{$D$-flat}
if $L_Z A = 0$ for all $Z \in \Gamma (D)$. It is not hard to see, for instance using local coordinates adapted to $D$, that $\widetilde M$ can be locally foliated by images of $D$-flat sections. In particular, for any $\alpha \in \widetilde M$ there exists a $D$-flat section $A$ of $\widetilde M$ locally defined around $x := \pi (\alpha)$ such that $A_x = \alpha$.
\begin{lemma}
Let $\alpha \in \widetilde M$. Moreover, let $A, A'$ be $D$-flat sections of $\widetilde M$ locally defined around $x := \pi (\alpha)$ such that $A_x = A'_x = \alpha$. Then $A_\ast (D_x) = A'_\ast (D_x)$.
\end{lemma}

\begin{proof}
Let $Z$ be a section of $D$. In particular $Z \in \mathfrak X_C$ and the natural lift of $Z$ to $\wedge^n T^\ast M$ is tangent to $\widetilde M \subset \wedge^{n}C^{0} \subset \wedge^n T^\ast M$. Denote by $\widetilde Z$ its restriction to $\widetilde M$. It follows from $L_Z A = L_Z A'$, that the flow of $\widetilde Z$ preserves both the image of $A$ of the image of $A'$. Hence, $A_\ast (Z_x) = \widetilde Z_\alpha = A'_\ast (Z_x)$.
\end{proof}

Now, let $\alpha \in \widetilde M$. Take a $D$-flat section $A$ of
$\widetilde{M}$, locally defined around $x=\pi(\alpha)$, such that
$A_{x}=\alpha$, and put $\widetilde{D}_{\alpha}:=A_{\ast}%
(D_{x})$. In view of the above lemma, $\widetilde D_\alpha$ is independent of the choice of $A$. Moreover, since $\widetilde M$ can be locally foliated by images of $D$-flat sections, the assignment $\widetilde D : \alpha \mapsto \widetilde D_\alpha$ is a regular distribution on $\widetilde M$, projecting pointwise isomorphically onto $D$. In particular $\widetilde D$ has the same (constant) rank as $D$.

I'm finally ready to prove the main result of this section.

\begin{theorem}
\label{8}The canonical form $\widetilde{\omega}$ on $\widetilde{M}$ is
pre-$n$-plectic and $\ker\widetilde{\omega}=\widetilde{D}$.
\end{theorem}

\begin{proof}
Notice preliminarily that vertical tangent vectors to $\wedge^{n}T^{\ast}M$
identify with points in $\wedge^{n}T^{\ast}M$. Similarly, vertical tangent
vectors to $\widetilde{M}$ at $\alpha$ identify with elements $v\in\wedge
^{n}T^{\ast}M$ such that $\ker v\supset C_{x}$, $x=\pi(\alpha)$. In the following,
I will understand such identifications. It is easy to see that 
\[
i_{v} \omega=\pi^{\ast}(v),
\] for all vertical tangent vectors $v$ to $\widetilde{M}$, where, in the rhs I interpret $v$ as a vertical vector, while, in the lhs I interpret $v$ as an $n$-form on $M$.

Now, let $A$ be a local $D$-flat section of $\widetilde{M}$ such that $A_{x}= \alpha$, so that $\widetilde{D}_{\alpha}:=A_{\ast}(D_{x})$. Pick $\zeta\in
D_{x}$ and $A_{\ast}(\zeta)\in\widetilde{D}_{\alpha}$. Show that $i_{A_{\ast}%
(\zeta)}\widetilde{\omega}=0$. It is enough to show that
\[
\widetilde{\omega}(v,A_{\ast}(\zeta),A_{\ast}(\xi_{1}),\ldots,A_{\ast}%
(\xi_{n-1}))=\widetilde{\omega}(A_{\ast}(\zeta),A_{\ast}(\xi_{1}%
),\ldots,A_{\ast}(\xi_{n}))=0
\]
for all vertical tangent vectors $v$ to $\widetilde{M}$ at $\alpha$, and $\xi
_{1},\ldots,\xi_{n}\in T_{x}M$. Now,
\[
\widetilde{\omega}(v,A_{\ast}(\zeta),A_{\ast}(\xi_{1}),\ldots,A_{\ast}%
(\xi_{n-1}))=v(\zeta,\xi_{1},\ldots,\xi_{n-1})=0
\]
since $\zeta\in D_{x}\subset C_{x}$. Moreover, let $Z \in \Gamma (D)$ be such that $Z_x = \zeta$. Then
\[
\widetilde{\omega}(A_{\ast}(\zeta),A_{\ast}(\xi_{1}),\ldots,A_{\ast}(\xi
_{n}))=dA(\zeta,\xi_{1},\ldots,\xi_{n})= (L_Z A)(\xi_{1},\ldots,\xi
_{n})=0.
\]

Conversely, let $\eta\in T_{\alpha}\widetilde{M}$ be such that $i_{\eta}%
\widetilde{\omega}$. In other words,
\[
\widetilde{\omega}(\eta,\eta_{1},\ldots,\eta_{n})=0\quad\text{for all }%
\eta_{1},\ldots,\eta_{n}\in T_{\alpha}\widetilde{M}.
\]
Put $\xi=\pi_{\ast}(\eta)$, and $\xi_{i}=\pi_{\ast}(\eta_{i})$,
$i=1,\ldots,n$. Then $\eta=A_{\ast}(\xi)+v$ and $\eta_{i}=A_{\ast}(\xi
_{i})+v_{i}$ for $v,v_{i}$ vertical tangent vectors to $\widetilde{M}$,
$i=1,\ldots,n$. Hence
\begin{align}
0  &  =\widetilde{\omega}(\eta,\eta_{1},\ldots,\eta_{n})\nonumber\\
&  =\widetilde{\omega}(A_{\ast}(\xi),\eta_{1},\ldots,\eta_{n})+v(\xi
_{1},\ldots,\xi_{n})\nonumber\\
&  =\widetilde{\omega}(A_{\ast}(\xi),A_{\ast}(\xi_{1}),\eta_{2},\ldots
,\eta_{n})-v_{1}(\xi,\xi_{2},\ldots,\xi_{n})+v(\xi_{1},\ldots,\xi_{n}). \label{2}%
\end{align}
Choosing $\xi_{1}=0$ one immediately see that $\xi\in C_{x}$, and (\ref{2})
simplifies to
\[
\widetilde{\omega}(A_{\ast}(\xi),A_{\ast}(\xi_{1}),\eta_{2},\ldots,\eta
_{n})+v(\xi_{1},\ldots,\xi_{n})=0.
\]
Now, choosing $\xi_{1}\in C_{x}$ and $v_{i}=0$ for $i>1$, one sees that, in
view, of Remark \ref{3}, $\xi\in D_{x}$. Finally, choose $v_{i}=0$ for $i>1$, and let $X \in \Gamma (D)$ be such that $\xi = X_x$
Then
\begin{align*}
0  &  =\widetilde{\omega}(A_{\ast}(\xi),A_{\ast}(\xi_{1}),\eta_{2},\ldots
,\eta_{n})+v(\xi_{1},\ldots,\xi_{n})\\
&  =dA(\xi,\xi_{1},\ldots,\xi_{n})+v(\xi_{1},\ldots,\xi_{n})\\
&  =(L_X A)(\xi_{1},\ldots,\xi_{n})+v(\xi_{1},\ldots,\xi_{n})\\
&  =v(\xi_{1},\ldots,\xi_{n}).
\end{align*}
Hence $v=0$. Concluding, $\eta=A_{\ast}(\xi)$, with $\xi\in D_{x}$.
\end{proof}

In particular $\widetilde{\omega}$ is a multisymplectic structure iff $C$ is a
multicontact structure.

\begin{definition}
The pair $(\widetilde{M},\widetilde{\omega})$ is called the
\emph{(pre-)multisymplectization of} $(M,C)$.
\end{definition}

\begin{remark}
[coordinate formulas]On $\widetilde{M}$ one can choose coordinates
$(x^{i},z^{a},p)$, where $p$ is implicitly defined by
$\alpha=p(\alpha)\Theta\in\widetilde{M}$. Notice that $p\neq0$. Locally, $\widetilde
{\theta}=p\Theta$,
\[
\widetilde{\omega}=\left(  dp+p\,\partial_{a}C_{i}^{a}dx^{i}\right)
\wedge\Theta-\dfrac{1}{2}pR_{ij}^{a}dx^{i}\wedge dx^{j}\wedge\Theta_{a},
\]
and a direct computation shows that $\widetilde{D}$ is locally generated by
vector fields of the form
\[
Z^{i}\left(  C_{i}-p\,\partial_{a}C_{i}^{a}\dfrac{\partial}{\partial
p}\right)  ,\quad\text{with }R_{ij}^{a}Z^{j}=0.
\]

\end{remark}

Now, I discuss the relationship between contact vector fields of $(M,C)$ and
locally Hamiltonian vector fields on the pre-multisymplectization. Let $X$ be
a \emph{multicontact vector field}, i.e.~$X\in\mathfrak{X}_{C}$. Then $X$ can be
naturally lifted to a vector field $\widetilde{X}$ on $\widetilde{M}$. Namely,
$X$ lifts to a unique vector field $X^{\ast}$ on $\wedge^{n}T^{\ast}M$
preserving the tautological $n$-form. It follows from multicontactness that
$X^{\ast}$ is actually tangent to $\wedge^{n}C^{0}$. In particular, it
restricts to a vector field $\widetilde{X}$ on $\widetilde{M}$ (which is open
in $\wedge^{n}C^{0}$). Clearly, $\widetilde{X}$ preserves the tautological
$n$-form on $\widetilde{M}$ and, therefore, it preserves $\widetilde{\omega}$.
Conversely, let $Y$ be a locally Hamiltonian vector field on $(\widetilde
{M},\widetilde{\omega})$. The next proposition shows, in particular, that, if
$Y$ is projectable onto $M$, then its projection is a multicontact vector field.

\begin{proposition}
\label{10}Let $Y$ be a locally Hamiltonian vector field on $(\widetilde{M},\widetilde
{\omega})$. If $Y$ projects on a vector field $X$ on $M$, then $X$ is
multicontact. Moreover, if $R\neq0$, then $Y=\widetilde{X}$.
\end{proposition}

\begin{proof}
I need some preliminary remarks. The Euler vector field $\Delta$ on
$\wedge^{n}C^{0}$ restricts to $\widetilde{M}$. Denote again by $\Delta$
the restriction. It is the fundamental vector field corresponding to the
canonical generator $1$ of the Lie algebra $\mathbb{R}$ of the structure group
$\mathbb{R}^{\times}$ of the principal bundle $\pi:\widetilde{M}%
\to M$. Locally $\Delta=p\partial/\partial p$. Notice that
\[
i_{\Delta}\widetilde{\theta}=0,\quad\text{and\quad}L_{\Delta}\widetilde
{\theta}=\widetilde{\theta},
\]
hence,
\[
i_{\Delta}\widetilde{\omega}=\widetilde{\theta},\quad\text{and\quad}L_{\Delta
}\widetilde{\omega}=\widetilde{\omega}.
\]
Now, let $Y$ be as in the statement and prove that $X$ is multicontact. It is enough to show that $Y$ preserves
the distribution $\widetilde{C}:=\pi_{\ast}^{-1}(C)$ on $\widetilde{M}$.
Clearly, $\widetilde{C}=\ker\widetilde{\theta}$. Now, since $Y$ is
projectable, then $[Y,\Delta]=f\Delta$ for some function $f$ on $\widetilde
{M}$. From $L_{Y}\widetilde{\omega}=0$, it follows
\[
0=i_{\Delta}L_{Y}\widetilde{\omega}=i_{[\Delta,Y]}\widetilde{\omega}%
+L_{Y}i_{\Delta}\widetilde{\omega}=-f\widetilde{\theta}+L_{Y}\widetilde
{\theta}.
\]
This shows that $L_{Y}\widetilde{\theta}=f\widetilde{\theta}$. Finally, let
$Z\in\Gamma(\widetilde{C})$. Compute
\[
i_{[Y,Z]}\widetilde{\theta}=[L_{Y},i_{Z}]\widetilde{\theta}=i_{Z}%
L_{Y}\widetilde{\theta}=fi_{Z}\widetilde{\theta}=0.
\]
This shows that $[Y,Z]\in\Gamma(\widetilde{C})$ for all $Z \in \Gamma (C)$, hence $X$ is
multicontact. In particular, $Y-\widetilde{X}$ is a vertical locally
Hamiltonian vector field. But, if $R\neq0$, then any vertical locally
Hamiltonian vector field $V$ is trivial. Indeed, $V=g\Delta$ for some function
$g$, and
\[
0=L_{g\Delta}\widetilde{\omega}=g\widetilde{\omega}+dg\wedge\widetilde{\theta
}.
\]
Assume by absurd that $g\neq0$ somewhere, hence in an open subset $U$ of
$\widetilde{M}$. Without loss of generality, let $g>0$ in $U$. Then
\begin{equation}
\widetilde{\omega}=dh\wedge\widetilde{\theta},\quad h=-\log g. \label{22}%
\end{equation}
Compute
\begin{equation}
i_{C_{i}}\widetilde{\omega}=p\,\partial_{a}C_{i}^{a}\Theta-pR_{ij}^{a}%
dz^{j}\wedge\Theta_{a}. \label{9}%
\end{equation}
But
\[
i_{C_{i}}(dh\wedge\widetilde{\theta})=C_{i}(h)\widetilde{\theta}%
=p\,C_{i}(h)\Theta,
\]
which is inconsistent with (\ref{9}) when $R\neq0$.
\end{proof}

\begin{remark}
The second part of Proposition \ref{10} cannot be extended to the case $R=0$.
Indeed, when $R=0$, in view of Frobenius Theorem, one can choose $C_{i}^{a}%
=0$, and $\widetilde{\omega}$ is locally given by $\widetilde{\omega}=dp\wedge
dz^{1}\wedge\cdots\wedge dz^{n}$. The vector field $\partial/\partial p$ is
locally Hamiltonian and vertical, in particular projectable. However, it
projects onto the trivial vector field.
\end{remark}

\begin{remark}
[coordinate formulas]Let $X$ be a multicontact vector field on $(M,C)$ locally
given by $X=X^{a}\partial_{a}+X^{i}C_{i}$. A direct computation shows that%
\begin{equation}
\widetilde{X}=X^{a}\partial_{a}+X^{i}C_{i}-(\partial_{a}X^{a}+\partial
_{a}C_{i}^{a}\,X^{i})\Delta. \label{26}%
\end{equation}

\end{remark}

\section{Homogeneous de Rham complex\label{SecHom}}

In the algebra $\Omega(\widetilde{M})$ of differential forms on $\widetilde
{M}$ consider the subspace $\Omega_{\circ}(\widetilde{M})$ consisting of
\emph{homogeneous differential forms}, i.e.~those differential forms $\mu$
such that $L_{\Delta}\mu=\mu$, $\Delta$ being the Euler vector field of
$\wedge^{n}C^{0}$ restricted to $\widetilde{M}$ (see the proof of
Proposition \ref{10}).

\begin{remark}
Differential forms $\widetilde{\theta}$ and $\widetilde{\omega}$ are homogeneous.
More generally, a form $\mu$ on $\widetilde{M}$ is homogeneous iff it is locally of the
form
\[
\mu=p\,\pi^{\ast}(\mu^{\prime})+dp\wedge\pi^{\ast}(\mu^{\prime\prime
}),\quad\mu^{\prime},\mu^{\prime\prime}\in\Omega(M).
\]
In particular, \end{remark}

In general, $\Omega_{\circ}(\widetilde{M})$ is not a subalgebra of
$\Omega(\widetilde{M})$. Nonetheless, since Lie derivatives commute with the
exterior differential, it is a subcomplex of the de Rham complex
$(\Omega(\widetilde{M}),d)$.

\begin{proposition}
\label{23}The \emph{homogenous de Rham complex} $(\Omega_{\circ}(\widetilde
{M}),d)$ is acyclic.
\end{proposition}

\begin{proof}
The insertion $i_{\Delta}$ is a contracting homotopy for $(\Omega_{\circ
}(\widetilde{M}),d)$, i.e.~$[i_{\Delta},d]=L_{\Delta}=\mathrm{id}$ on
$\Omega_{\circ}(\widetilde{M})$.
\end{proof}

It immediately follows from the proof of the above proposition that
\begin{equation}
\Omega_{\circ}^{n-1}(\widetilde{M})=B_{\circ}\oplus K \label{4}%
\end{equation}
where
\begin{align*}
B_{\circ}  &  :=\operatorname{im}(d:\Omega_{\circ}^{n-2}(\widetilde
{M})\longrightarrow\Omega_{\circ}^{n-1}(\widetilde{M})),\\
K  &  :=\ker(i_{\Delta}:\Omega_{\circ}^{n-1}(\widetilde{M})\longrightarrow
\Omega_{\circ}^{n-2}(\widetilde{M})).
\end{align*}
Decomposition (\ref{4}) is implemented as follows. For $\mu\in\Omega
_{\circ}^{n-1}(\widetilde{M})$,
\[
\mu=L_{\Delta}\mu=di_{\Delta}\mu+i_{\Delta}d\mu
\]
with $di_{\Delta}\mu\in B_{\circ}$, and $i_{\Delta}d\mu\in K$.

Now, recall that an $(n-1)$-form $\mu$ on $\widetilde{M}$ is
\emph{Hamiltonian} iff there is a vector field $X$, an associated
\emph{Hamiltonian vector field}, not necessarily unique (unless $\widetilde
{\omega}$ is multisymplectic), such that $i_{X}\omega=-d\mu$. Denote by
$\Omega_{\mathrm{Ham}}^{n-1}(\widetilde{M},\widetilde{\omega})\subset
\Omega^{n-1}(\widetilde{M})$ the vector subspace of Hamiltonian forms.
Consider the distinguished subspace $\Omega_{\mathrm{Ham}}^{n-1}(M,C)$ of
$\Omega_{\mathrm{Ham}}^{n-1}(\widetilde{M},\widetilde{\omega})$ consisting of
homogenous Hamiltonian forms with an associated Hamiltonian vector field which
is projectable on $M$:
\[
\Omega_{\mathrm{Ham}}^{n-1}(M,C):=\{\mu\in\Omega_{\circ}^{n-1}%
(\widetilde{M}):\exists Y\in\mathfrak{X}(\widetilde{M})\text{ such that
}Y\text{ is projectable and }i_{Y}\widetilde{\omega}=-d\mu\}.
\]
As I will show in the next section elements in $\Omega_{\mathrm{Ham}}%
^{n-1}(M,C)$ should be understood as the contact analogues of Hamiltonian
forms on (pre-)multisymplectic manifolds. For now, notice that if $C$ is
$1$-contact, then $\widetilde{\omega}$ is $1$-plectic and every homogeneous
function on $\widetilde{M}$ is in $\Omega_{\mathrm{Ham}}^{n-1}(M,C)$.
Moreover, since $\widetilde{M}$ is dense in $C^{0}$, homogeneous functions
on $\widetilde{M}$ identify with fiberwise linear functions on $C^{0}$,
i.e.~sections of $N$. Thus, if $C$ is $1$-contact, $\Omega_{\mathrm{Ham}%
}^{n-1}(M,C)$ is in one-to-one correspondence with $\Gamma(N)$, sections of
the Jacobi bundle. In the general case, I will be interested in the following
truncated, homogenous de Rham complex
\begin{equation}
0\longrightarrow C_{\circ}^{\infty}(\widetilde{M})\overset{d}{\longrightarrow
}\Omega_{\circ}^{1}(\widetilde{M})\overset{d}{\longrightarrow}\cdots
\longrightarrow\Omega_{\circ}^{n-2}(\widetilde{M})\overset{d}{\longrightarrow
}\Omega_{\mathrm{Ham}}^{n-1}(M,C)\longrightarrow0, \label{5}%
\end{equation}
which is obviously well defined. In view of Proposition (\ref{23}) the
cohomology of (\ref{5}) is trivial everywhere except in the last term where it
is $\Omega_{\mathrm{Ham}}^{n-1}(M,C)/B_{\circ}$. The next proposition
describes the quotient $\Omega_{\mathrm{Ham}}^{n-1}(M,C)/B_{\circ}$.

\begin{proposition}
\label{13}There is a canonical isomorphism
\[
\Gamma_{\mathrm{Ham}}(N)\simeq\Omega_{\mathrm{Ham}}^{n-1}(M,C)/B_{\circ}.
\]

\end{proposition}

\begin{proof}
It follows from (\ref{4}) that
\[
\Omega_{\mathrm{Ham}}^{n-1}(M,C)=B_{\circ}\oplus K_{\mathrm{Ham}}%
\]
where $K_{\mathrm{Ham}}:=K\cap\Omega_{\mathrm{Ham}}^{n-1}(M,C)$. It is then
enough to show that $K_{\mathrm{Ham}}\simeq\Gamma_{\mathrm{Ham}}(N)$. The
isomorphism can be described as follows. First of all, notice that, for all
$a\in\widetilde{M}$, since $\ker a=C_{x}$, $x=\pi(a)$, there exists a unique
linear map $\varphi_{a}:N_{x}\to\wedge^{n-1}T_{x}^{\ast}M$ such
that $a=\varphi_{a}\circ\theta$. Now, let $\nu\in\Gamma(N)$. Define an
$(n-1)$-form $\widetilde{\nu}$ on $\widetilde{M}$ by putting
\[
\widetilde{\nu}_{a}:=\pi^{\ast}(\varphi_{a}(\nu)),\quad a\in\widetilde{M}.
\]
If $\nu$ is locally given by $\nu=\nu^{a}\theta(\partial_{a})$, then
$\widetilde{\nu}$ is locally given by $\widetilde{\nu}=p\nu^{a}\Theta_{a}$.
This shows that $\widetilde{\nu}\in K$. Moreover $\widetilde{\nu}=0$ iff
$\nu=0$. Finally, $\widetilde{\nu}\in K_{\mathrm{Ham}}$ if $\nu\in
\Gamma_{\mathrm{Ham}}(N)$. Indeed, let $\nu$ be a Hamiltonian section. Then
$\nu=\theta(X)$ for some multicontact vector field $X$. Lift it to a
multisymplectic vector field $\widetilde{X}$ on $\widetilde{M}$. It is easy to
see, for instance in local coordinates, that $i_{\widetilde{X}}\widetilde
{\theta}=\widetilde{\nu}$. It follows that
\[
i_{\widetilde{X}}\widetilde{\omega}=i_{\widetilde{X}}d\widetilde{\theta
}=L_{\widetilde{X}}\widetilde{\theta}-di_{\widetilde{X}}\widetilde{\theta
}=-d\widetilde{\nu}.
\]
Define the injective map $\Gamma_{\mathrm{Ham}}(N)\to
K_{\mathrm{Ham}}$ as $\nu\mapsto\widetilde{\nu}$. Conversely, let
$\mu\in K_{\mathrm{Ham}}$ and $Y$ be an associated projectable Hamiltonian
vector field, i.e.~$i_{Y}\widetilde{\omega}=-d\mu$. Then,
\[
0=i_{\Delta}(i_{Y}\widetilde{\omega}+d\mu)=-i_{Y}i_{\Delta}\widetilde
{\omega}+L_{\Delta}\mu=-i_{Y}\widetilde{\theta}+\mu,
\]
i.e.~$\mu=i_{Y}\widetilde{\theta}$. If $Y$ is locally given by
$Y=Y^{a}\partial_{a}+Y^{i}C_{i}+Y_{0}\Delta$, then $\mu=i_{Y}%
\widetilde{\theta}=pY^{a}\Theta_{a}$. Now, let $X$ be the projection of $Y$.
In view of Proposition \ref{10}, $X$ is multicontact and $Y=\widetilde{X}$.
Moreover, $\theta(X)=Y^{a}\theta(\partial_{a})$. Hence $\mu=\widetilde
{\theta(X)}$. This shows that the map $\Gamma_{\mathrm{Ham}}(N)\to
K_{\mathrm{Ham}}$ defined above is surjective.
\end{proof}

\begin{remark}
If $C$ is multicontact, then $\widetilde{\omega}$ is multisymplectic and
\[
\Omega_{\mathrm{Ham}}^{n-1}(M,C)=\Omega_{\mathrm{Ham}}^{n-1}(\widetilde
{M},\widetilde{\omega})\cap\Omega_{\circ}^{n-1}(\widetilde{M}).
\]
Indeed, $\Omega_{\mathrm{Ham}}^{n-1}(M,C)\subset\Omega_{\mathrm{Ham}}%
^{n-1}(\widetilde{M},\widetilde \omega)\cap\Omega_{\circ}^{n-1}(\widetilde{M})$.
On the other hand, let $\mu\in\Omega_{\mathrm{Ham}}^{n-1}(\widetilde{M},\widetilde{C})$
be homogeneous, and let $Y$ be the (unique) Hamiltonian vector field
associated to it. Then%
\[
0=L_{\Delta}(i_{Y}\widetilde{\omega}+d\mu)=i_{[\Delta,Y]}\widetilde{\omega
}.
\]
Hence $[\Delta,Y]=0$ and $Y$ is projectable. This shows that $\mu\in
\Omega_{\mathrm{Ham}}^{n-1}(M,C)$. Thus, $\Omega_{\mathrm{Ham}}^{n-1}%
(\widetilde{M},\widetilde{\omega})\cap\Omega_{\circ}^{n-1}(\widetilde
{M})\subset\Omega_{\mathrm{Ham}}^{n-1}(M,C)$.
\end{remark}

\section{$L_{\infty}$-algebras from multicontact geometry}

In this section I define a {multicontact} analogue of the Jacobi bundle of a contact
manifold. Equivalently, I define the contact analogue of the $L_{\infty}%
$-algebra of a (pre-)multisymplectic manifold \cite{r12,z12}. First of all
recall the definition of an $L_{\infty}$-algebra. I use the ``
homological convention'' . In what follows I denote by $|v|$ the degree of a homogeneous element $v$ in a graded vector space.

\begin{definition}
[\cite{ls93,lm95}]An $L_{\infty}$\emph{-algebra} is a pair $(\mathfrak{g}%
,\{\lambda_{\ell},\ \ell\in\mathbb{N}\})$, where $\mathfrak{g}=\bigoplus_{i}$
$\mathfrak{g}_{i}$ is a graded vector space, and the $\lambda_{\ell}$'s are
$\ell$-ary, graded, multilinear, degree $\ell-2$ operations
\[
\lambda_{\ell}:{}\wedge^{\ell}\mathfrak{g}\longrightarrow\mathfrak{g},\quad
k\in\mathbb{N},
\]
such that
\begin{equation}
\sum_{i+j=\ell}(-)^{ij}\sum_{\sigma}\chi(\sigma,\boldsymbol{v})\,\lambda
_{j+1}(\lambda_{i}(v_{\sigma(1)},\ldots,v_{\sigma(i)}),v_{\sigma(i+1)}%
,\ldots,v_{\sigma(i+j)})=0, \label{24}%
\end{equation}
for all $v_{1},\ldots,v_{\ell}\in\mathfrak{g}$, $\ell\in\mathbb{N}$ (in
particular, $(\mathfrak{g},\lambda_{1})$ is a chain complex and
$H(\mathfrak{g},\lambda_{1})$ is a graded Lie algebra).
\end{definition}

In Formula (\ref{24}), the sum is over all unshuffles $S_{i,j}$, i.e.,
permutations $\sigma$ of $\{1,\ldots,\ell\}$ such that $\sigma(1)<\cdots
<\sigma(i)$, and $\sigma(i+1)<\cdots<\sigma(\ell)$, and $\chi(\sigma
,\boldsymbol{v})$ is the sign implicitly defined by
\[
v_{\sigma(1)}\wedge\cdots\wedge v_{\sigma(\ell)}=\chi(\mu,\boldsymbol{v}%
)v_{1}\wedge\cdots\wedge v_{\ell},
\]
where the wedge $\wedge$ indicates the exterior (graded skew-symmetric)
product of elements in $\mathfrak{g}$, which satisfies, by definition,
$v\wedge w=-(-)^{|v||w|}w\wedge v$, for all homogeneous elements $v,w \in \mathfrak g$.

If $\mathfrak{g}$ is concentrated in degree $0$, then an $L_{\infty}$-algebra
structure on $\mathfrak{g}$ is simply a Lie algebra structure. Similarly, if
$\lambda_{\ell}=0$ for all $\ell>2$, then $(\mathfrak{g},\{\lambda_{\ell
},\ \ell\in\mathbb{N}\})$ is a differential graded Lie algebra. More
generally, $L_{\infty}$-algebras are \emph{Lie algebras up to homotopy}.
Indeed, the binary bracket $\lambda_{2}$ of an $L_{\infty}$-algebra satisfies
the (graded) Jacobi identity only up to an homotopy encoded by $\lambda_{3}$.
Similarly, the higher brackets satisfy higher versions of the Jacobi identity
(up to homotopies).

In \cite{r12} and \cite{z12} the authors show that there is an $L_{\infty}%
$-algebra canonically associated to a (pre-)multisymplectic manifold. Such
$L_{\infty}$-algebra plays a role analogous to that of the Poisson algebra of
functions on a symplectic manifold \cite{frz13,frs13b,frs13}. In the case of
the (pre-)multisymplectization $(\widetilde{M},\widetilde{\omega})$ of a
pre-$n$-contact manifold, Rogers and Zambon results read as follows.

\begin{theorem}
There is an $L_{\infty}$-algebra $\mathfrak{g}_{\bullet}(\widetilde
{M},\widetilde{\omega})=\bigoplus_{i=0}^{n-1}\mathfrak{g}_{i}(\widetilde
{M},\widetilde{\omega})$, concentrated in degrees $0,\ldots,n-1$, where
\[
\mathfrak{g}_{i}(\widetilde{M},\widetilde{\omega}):=\left\{
\begin{array}
[c]{ll}%
\Omega_{\mathrm{Ham}}^{n-1}(\widetilde{M},\widetilde{\omega}) & \text{if
}i=0\\
\Omega^{n-i-1}(\widetilde{M}) & \text{if }0<i\leq n-1
\end{array}
\right.  .
\]
The operations in $\mathfrak{g}_{\bullet}(\widetilde{M},\widetilde{\omega})$
are defined as follows $(\mathfrak{g}_{\bullet}(\widetilde{M},\widetilde
{\omega}),\lambda_{1})$ is the truncated de Rham complex
\[
0\longleftarrow\Omega_{\mathrm{Ham}}^{n-1}(\widetilde{M},\widetilde{\omega
})\overset{d}{\longleftarrow}\Omega^{n-2}(\widetilde{M})\longleftarrow
\cdots\overset{d}{\longleftarrow}\Omega^{1}(\widetilde{M})\overset
{d}{\longleftarrow}C^{\infty}(\widetilde{M})\longleftarrow0,
\]
and, for $\ell>0$,
\[
\lambda_{\ell}(\mu_{1},\ldots,\mu_{\ell})=\left\{
\begin{array}
[c]{cc}%
-(-)^{\ell}i_{X_{\mu_{1}}}\cdots i_{X_{\mu_{\ell}}}\widetilde{\omega} &
\text{if }| \mu_{1} |+\cdots+| \mu_{\ell} |=0\\
0 & \text{if }| \mu_{1} |+\cdots+| \mu_{\ell} |>0
\end{array}
\right.  ,
\]
where $X_{\mu}$ is an Hamiltonian vector field associated to the
Hamiltonian form $\mu$.
\end{theorem}

Elements of the $L_{\infty}$-algebra $\mathfrak{g}_{\bullet}(\widetilde
{M},\widetilde{\omega})$ should be interpreted as observables of
multisymplectic field theories defined on $(\widetilde{M},\widetilde{\omega})$
(see, \cite{bhr19,frs13b,frs13}). I will now present a contact analogue of
$\mathfrak{g}_{\bullet}(\widetilde{M},\widetilde{\omega})$. At the same time,
it should be a {multicontact version} of the
Jacobi bundle of a standard contact manifold. In order to motivate my
definition, I remark that sections of the Jacobi bundle of a contact manifold
$(M,C)$ can be understood as homogeneous functions on the symplectization
$(\widetilde{M},\widetilde{\omega})$ (see Section \ref{SecCont}). The Jacobi
bracket is then just the restriction to $C_{\circ}^{\infty}(\widetilde
{M})\simeq\Gamma(N)$ of the Poisson bracket on $(\widetilde{M},\widetilde
{\omega})$. This suggests to look for the contact analogue of $\mathfrak{g}%
_{\bullet}(\widetilde{M},\widetilde{\omega})$ on the (truncated) homogeneous
de Rham complex of Section \ref{SecHom}.

Propositions (\ref{23}) and (\ref{13}) show that the truncated homogeneous de
Rham complex (\ref{5}) provides a resolution of $\Gamma_{\mathrm{Ham}}(N)$. In
its turn, $\Gamma_{\mathrm{Ham}}(N)$ is a Lie algebra. In \cite{b...98}
Barnich, Fulp, Lada, and Stasheff proved that this situation is precisely a
source of $L_{\infty}$-algebras. Namely, whenever the underlying vector space
of a Lie algebra is resolved by a chain complex, then there is an $L_{\infty}%
$-algebra structure on chains such that 1) the unary operation agrees with the
differential, and 2) the binary bracket induces the Lie bracket in homology.
It immediately follows that there is an $L_{\infty}$-algebra structure on the
underlying graded vector space of (\ref{5}) such that 1) the unary operation
is the de Rham differential, and 2) the binary operation induces the Lie
bracket between Hamiltonian sections in homology. Actually, such $L_{\infty
}$-algebra can be described in terms of the $L_{\infty}$-algebra
$\mathfrak{g}_{\bullet}(\widetilde{M},\widetilde{\omega})$, at least in the
case $R\neq0$, as shown below.

\begin{proposition}
\label{27}If $R\neq0$, the operations $\lambda_{\ell}$ on $\mathfrak{g}%
_{\bullet}(\widetilde{M},\widetilde{\omega})$ restrict to the homogeneous
truncated de Rham complex.
\end{proposition}

\begin{proof}
Recall that $\widetilde{\omega}$ is itself homogeneous. Thus, it is enough to
prove that, whenever $\mu\in\Omega_{\mathrm{Ham}}^{n-1}(M,C)$, then the
insertion $i_{X_{\mu}}$ of an associated projectable Hamiltonian vector
field $X_{\mu}$ preserves homogenous forms. This
immediately follows from Proposition \ref{10}. Indeed, $X_{\mu}$ is the
locally Hamiltonian lift of a multicontact vector field on $(M,C)$ and,
therefore, $i_{X_{\mu}}$ has the required property (see, for instance,
coordinate Formula (\ref{26})).
\end{proof}

Let $R\neq0$. Collecting the above results, I get the following

\begin{theorem}\label{Theor}
There is an $L_{\infty}$-algebra $\mathfrak{g}_{\bullet}(M,C)=\bigoplus
_{i=0}^{n-1}\mathfrak{g}_{i}(M,C)$, concentrated in degrees $0,\ldots,n-1$,
where
\[
\mathfrak{g}_{i}(M,C):=\left\{
\begin{array}
[c]{ll}%
\Omega_{\mathrm{Ham}}^{n-1}(M,C) & \text{if }i=0\\
\Omega_{\circ}^{n-i-1}(\widetilde{M}) & \text{if }0<i\leq n-1
\end{array}
\right.  .
\]
The operations in $\mathfrak{g}_{\bullet}(M,C)$ are defined as follows
$(\mathfrak{g}_{\bullet}(M,C),\lambda_{1})$ is the truncated homogeneous de
Rham complex 
\begin{equation}\label{Res}
0\longleftarrow\Omega_{\mathrm{Ham}}^{n-1}(M,C)\overset{d}{\longleftarrow
}\Omega_{\circ}^{n-2}(\widetilde{M})\longleftarrow\cdots\overset
{d}{\longleftarrow}\Omega_{\circ}^{1}(\widetilde{M})\overset{d}{\longleftarrow
}C_{\circ}^{\infty}(\widetilde{M})\longleftarrow0,
\end{equation}
and, for $\ell>0$,
\[
\lambda_{\ell}(\mu_{1},\ldots,\mu_{\ell})=\left\{
\begin{array}
[c]{cc}%
-(-)^{\ell}i_{X_{\mu_{1}}}\cdots i_{X_{\mu_{\ell}}}\widetilde{\omega} &
\text{if }| \mu_{1} |+\cdots+| \mu_{\ell} |=0\\
0 & \text{if }| \mu_{1} |+\cdots+| \mu_{\ell} |>0
\end{array}
\right.  ,
\]
where $X_{\mu}$ is a projectable Hamiltonian vector field associated to the
homogeneous Hamiltonian form $\mu\in\Omega_{\mathrm{Ham}}^{n-1}(M,C)$.
\end{theorem}

Moreover, one has the

\begin{proposition}\label{proposition}
The binary operation in $\mathfrak{g}_{\bullet}(M,C)$ induces the Lie bracket
on $\Gamma_{\mathrm{Ham}}(N)$ in homology.
\end{proposition}

\begin{proof}
Let $\mu_{1},\mu_{2}\in\Omega_{\mathrm{Ham}}^{n-1}(M,C)$. Their binary
operation $\lambda_{2}(\mu_{1},\mu_{2})$ in $\mathfrak{g}_{\bullet
}(M,C)$ is a homogeneous, Hamiltonian form with a projectable Hamiltonian
vector field $[X_{\mu_{1}},X_{\mu_{2}}]$, where $X_{\mu_{1}%
},X_{\mu_{2}}$ are projectable Hamiltonian vector fields associated to
$\mu_{1},\mu_{2}$ respectively. Indeed, $[X_{\mu_{1}},X_{\mu_{2}%
}]$ is projectable, and, since $di_{X_{\mu_{2}}}\widetilde{\omega}%
=dd\mu_{2}=0$, and $L_{X_{\mu_{1}}}\widetilde{\omega}=0$, one gets
\[
-d\lambda_{2}(\mu_{1},\mu_{2})=di_{X_{\mu_{1}}}i_{X_{\mu_{2}}%
}\widetilde{\omega}=L_{X_{\mu_{1}}}i_{X_{\mu_{1}}}\widetilde{\omega
}=[L_{X_{\mu_{1}}},i_{X_{\mu_{1}}}]\widetilde{\omega}=i_{[X_{\mu_{1}%
},X_{\mu_{2}}]}\widetilde{\omega}.
\]
Thus, the homology class of $\lambda_{2}(\mu_{1},\mu_{2})$ in the
truncated, homogeneous de Rham complex identify with $\theta(\pi_{\ast
}[X_{\mu_{1}},X_{\mu_{2}}])$ (see the proof of Proposition \ref{13}),
which is given by
\[
\theta(\pi_{\ast}[X_{\mu_{1}},X_{\mu_{2}}])=\theta([\pi_{\ast}%
X_{\mu_{1}},\pi_{\ast}X_{\mu_{2}}])=[\theta(\pi_{\ast}X_{\mu_{1}%
}),\theta(\pi_{\ast}X_{\mu_{2}})].
\]

\end{proof}

\begin{remark}
As already remarked, complex (\ref{Res}) is actually a resolution of $\Gamma_\mathrm{Ham} (N)$. In other words, there is a quasi-isomorphism of complexes:
\[
\xymatrix{0 & \mathfrak{g}_0 (M,C) \ar[d] \ar[l]& \mathfrak{g}_1 (M,C) \ar[l] \ar[d] & \cdots \ar[l] & \mathfrak{g}_n (M,C)  \ar[d]\ar[l] & 0 \ar[l] \\
               0 & \Gamma_{\mathrm{Ham} (N)} \ar[l] & 0 \ar[l] & \cdots \ar[l] & 0 \ar[l] & 0 \ar[l]
}
\]
Now, Proposition \ref{proposition} implies that such quasi-isomorphism extends to an $L_\infty$-quasi-isomorphism so that the $L_\infty$-algebra $\mathfrak{g}_\bullet (M,C)$ and the Lie algebra $\Gamma_\mathrm{Ham} (N)$ are isomorphic objects in the homotopy category of $L_\infty$-algebras (see, e.g., \cite{v12} for an introduction to the homotopy theory of $L_\infty$-algebras). In particular, from the point of view of homotopical algebra, $\mathfrak{g}_\bullet (M,C)$ contains precisely the same information as $\Gamma_\mathrm{Ham} (N)$.  Hence, the situation is different from that in the multisymplectic case. Namely the Rogers and Zambon $L_\infty$-algebra is not $L_\infty$-quasi-isomorphic to the Lie algebra of Hamiltonian vector fields in general, and, therefore, its quasi-isomorphism class contains more information.

Exploring the applications of the multicontact $L_\infty$-algebra goes beyond the scopes of this paper and the ultimate utility of ``replacing $\Gamma_\mathrm{Ham} (N)$ with $\mathfrak{g}_\bullet (M,C)$'' remains an open question. In this respect, I can only add few more (intentionally vague) words. Every system of (possibly nonlinear) partial differential equations (PDE) can be geometrically understood as a manifold with a distribution $(\mathcal E, C)$, i.e.~a pre-multicontact manifold (up to conceptually irrelevant regularity issues). Thus, to a system of PDEs, one can attach an $L_\infty$-algebra $\mathfrak g (\mathcal E, C)$ encoding infinitesimal symmetries of the system. There is another, completely different, $L_\infty$-algebra that can be attached to a system of PDEs as announced in \cite{v14}. Denote it by $\mathfrak g_\infty$. The $L_\infty$-algebra $\mathfrak g_\infty$ encodes, among other data, so called \emph{higher} infinitesimal symmetries of $(\mathcal E, C)$ \cite{b...99}. It is an interesting open issue ``whether or not $\mathfrak g (\mathcal E, C)$ and $\mathfrak g_\infty$ interact or not'', and, if yes, whether or not one can extract new information on $(\mathcal E, C)$ from this interaction.
\end{remark}

\subsection*{Acknowledgments}
I thank the anonymous referee for carefully reading the manuscript and for suggesting several revisions that improved the exposition.

\appendix

\section{Lie Algebroids and An Alternative Description of Homogeneous
Forms\label{SecLie}}
{In this appendix I sketch an alternative description of the $L_\infty$-algebra of a pre-multicontact manifold, independent of the multisymplectization. Such description exploits the technology of Lie algebroids and their representation.}

Recall that a \emph{Lie algebroid} over a manifold $M$ is a vector bundle
$A\to M$ equipped with 1) a $C^{\infty}(M)$-linear map
$\varrho:\Gamma(A)\to\mathfrak{X}(M)$ called the \emph{anchor},
and 2) a Lie bracket $[-,-]$ on $\Gamma(A)$ such that
\[
\lbrack\alpha,f\beta]=\varrho(\alpha)(f)\beta+f[\alpha,\beta],\quad
\alpha,\beta\in\Gamma(A),\quad f\in C^{\infty}(M).
\]


\begin{example}
\label{17}Let $E\to M$ be a vector bundle. An $\mathbb{R}$-linear
operator $\square:\Gamma(E)\to\Gamma(E)$ is a \emph{derivation of
}$\Gamma(E)$ if there exists a (necessarily unique) vector field
$s(\square)$ on $M$, sometimes called the \emph{symbol of} $\square$,
such that
\[
\square(f\varepsilon)=s(\square)(f)\varepsilon+f\square(\varepsilon
),\quad\varepsilon\in\Gamma(E),\quad f\in C^{\infty}(M).
\]
Denote by $\operatorname{Der}E$ the space of derivations of $\Gamma(E)$. It is
a $C^{\infty}(M)$-module with the obvious multiplication, and, a Lie algebra
with Lie bracket given by the commutator. Even more, it is the module, and Lie
algebra of sections of a Lie algebroid $\operatorname{der}E\to M$,
with anchor $\operatorname{Der}E\to\mathfrak{X}(M)$, given by
$\square\mapsto s(\square)$. There is a more geometric description of
derivations of $\Gamma(E)$. Namely, denote by $\mathfrak{X}_{\mathrm{lin}%
}(E^{\ast})\subset\mathfrak{X}(E^{\ast})$ the subspace of \emph{linear vector
fields on }$E^{\ast}$, i.e.~those vector fields on the dual bundle $E^{\ast}$
preserving fiberwise linear functions on $E^{\ast}$. Linear vector fields are
projectable over $M$. For $X\in\mathfrak{X}_{\mathrm{lin}}(E^{\ast})$ denote
by $\underline{X}\in\mathfrak{X}(M)$ its projection. It is easy to see that
$\mathfrak{X}_{\mathrm{lin}}(E^{\ast})$ is a $C^{\infty}(M)$-submodule and a
Lie subalgebra of $\mathfrak{X}(E^{\ast})$. Now, fiberwise linear functions on
$E^{\ast}$ identify with sections of $E$, and the map $\mathfrak{X}%
_{\mathrm{lin}}(E^{\ast})\to\operatorname{Der}E$, $X\mapsto
X|_{\Gamma(E)}$ is an isomorphism of $C^{\infty}(M)$-modules and of Lie
algebras, such that $\underline{X}= s(X|_{\Gamma(E)})$.
\end{example}

Let $A\to M$ be a Lie algebroid. Recall that a
\emph{representation of} $A$ is a vector bundle $V\to M$ equipped
with a \emph{flat }$A$\emph{-connection} $\nabla$, i.e.~a $C^{\infty}%
(M)$-linear map $\nabla:\Gamma(A)\to\operatorname{Der}V$, denoted
$\alpha\mapsto\nabla_{\alpha}$, such that 1) $ s(\nabla_{\alpha
})=\varrho(\alpha)$, and 2) $[\nabla_{\alpha},\nabla_{\beta}]=\nabla
_{\lbrack\alpha,\beta]}$, $\alpha,\beta\in\Gamma(A)$. Let $(V,\nabla)$ be a
representation of $A$. The graded vector space $\operatorname{Alt}%
(A,V):=\operatorname{Alt}(\Gamma(A),\Gamma(V))$ of alternating, $C^{\infty
}(M)$-multilinear, $\Gamma(V)$-valued forms on $\Gamma(A)$ is naturally
equipped with a homological operator $d_{A}$ given by the following
\emph{Chevalley-Eilenberg formula}:
\begin{align*}
&  (d_{A}\varpi)(\alpha_{1},\ldots,\alpha_{k+1})\\
&  :=\sum_{i}(-)^{i}\nabla_{\alpha_{i}}(\varpi(\ldots,\widehat{\alpha_{i}%
},\ldots))+\sum_{i<j}(-)^{i+j}\varpi([\alpha_{i},\alpha_{j}],\ldots
,\widehat{\alpha_{i}},\ldots,\widehat{\alpha_{j}},\ldots),
\end{align*}
where $\varpi\in\operatorname{Alt}^{k}(\Gamma(A),\Gamma(V))$ is an alternating
form with $k$-entries, $\alpha_{1},\ldots,\alpha_{k+1}\in\Gamma(A)$, and a hat
$\widehat{\left(  -\right)  }$ denotes omission.

\begin{example}
\label{25}There is a \emph{tautological representation }of $\operatorname{der}%
E\to M$, given by $(E,\nabla)$, with structure flat connection
$\nabla:\operatorname{Der}E\to\operatorname{Der}E$ being the
identity, i.e.~$\nabla_{\square}\varepsilon=\square(\varepsilon)$,
$\square\in\operatorname{Der}E$, $\varepsilon\in\Gamma(E)$. The associated
complex $(\operatorname{Alt}(\operatorname{der}E,E),d_{\operatorname{der}E})$
is sometimes called the ($E$\emph{-valued}) \textrm{Der}\emph{-complex}
\cite{r80}. Let $\Delta$ be the Euler vector field on $E^{\ast}$. One can
define a subcomplex $(\Omega_{\circ}(E^{\ast}),d)$ of the de Rham complex of
$E^{\ast}$ exactly as in Section \ref{SecHom}. Namely, a differential form
$\omega$ on $E^{\ast}$ is in $\Omega_{\circ}(E^{\ast})$ if $L_{\Delta}%
\omega=\omega$. In particular elements in $C_{\circ}^{\infty}(E^{\ast})$ are
fiberwise linear functions on $E^{\ast}$, i.e.~sections of $E$. Elements of
$\Omega_{\circ}(E^{\ast})$ are called \emph{linear differential forms} on
$E^{\ast}$ and are preserved by the exterior differential. There is a
canonical embedding of graded vector spaces
\begin{equation}
\iota:\Omega_{\circ}(E^{\ast})\to\operatorname{Alt}%
(\operatorname{der}E,E), \label{11}%
\end{equation}
which can be defined as follows. Denote by $\varphi:\mathfrak{X}%
_{\mathrm{lin}}(E^{\ast})\to\operatorname{Der}E$, $X\mapsto
X|_{\Gamma(E)}$ the isomorphism of Example \ref{17}. Notice that $\Delta
\in\mathfrak{X}_{\mathrm{lin}}(E^{\ast})$ and $\varphi(\Delta)$ is the
identity of $\Gamma(E^{\ast})$. Moreover, for $X\in\mathfrak{X}_{\mathrm{lin}%
}(E^{\ast})$, the insertion $i_{X}$ maps linear differential forms to linear
differential forms. Thus, let $\mu\in\Omega_{\circ}^{k}(E^{\ast})$. Define
$\iota(\mu)\in\operatorname{Alt}^{k}(\operatorname{Der}E,\Gamma(E))$ by
putting
\[
\iota(\mu)(\square_{1},\ldots,\square_{k}):=i_{\varphi^{-1}(\square_{k}%
)}\cdots i_{\varphi^{-1}(\square_{1})}\mu,\quad\square_{1},\ldots
,\square_{k}\in\operatorname{Der}E
\]
It is easy to see that $\iota$ is injective. Moreover, it is a cochain map.
Finally, dimension counting proves that\emph{ }$\iota$\emph{ is also
surjective when }$E$\emph{ is a line bundle}.
\end{example}

The above example provides an alternative description of the homogenous de
Rham complex of Section \ref{SecHom}. Indeed, use the same notations as in
Section \ref{SecHom}. Since $\widetilde{M}$ is dense in $\wedge^{n}C^{0}$,
then the restriction of linear forms on $\wedge^{n}C^{0}$ to homogeneous
forms on $\widetilde{M}$ is an isomorphism. Moreover, $\wedge^{n}C^{0
}\simeq L^{\ast}$, where $L:=\wedge^{n}N$ is a line bundle. Collecting
previous remarks one gets
\[
(\Omega_{\circ}(\widetilde{M}),d)\simeq(\operatorname{Alt}(\operatorname{der}%
L,L),d_{\operatorname{der}L}).
\]
Finally, recall that $\widetilde{\theta}\in\Omega_{\circ}^{n}(\widetilde{M})$.
It is easy to see that the corresponding element $\iota(\widetilde{\theta})$
in $\operatorname{Alt}(\operatorname{der}L,L)$ is given by
\begin{equation}
\iota(\widetilde{\theta})(\square_{1},\ldots,\square_{n}):=\theta
( s(\square_{1}))\wedge\cdots\wedge\theta( s(\square_{n}%
)),\quad\square_{1},\ldots,\square_{n}\in\operatorname{der}L.\label{18}%
\end{equation}
Using formula (\ref{18}) one can also find $\iota(\widetilde{\omega}%
)=\iota(d\widetilde{\theta})=d_{\operatorname{der}L}\iota(\widetilde{\theta}%
)$, and describe the higher brackets in $\mathfrak{g}(M,C)$ without reference
to the multisymplectization. Details are left to the reader.

\section{Jet spaces and $L_{\infty}$-algebras\label{SecJet}}

In this appendix, I provide explicit coordinate formulas for the $L_{\infty}%
$-algebras determined by the canonical multicontact structures on jet spaces.
Let $E$ be a $(n+m)$-dimensional manifold, and let $J^{k}=J^{k}(E,m)$ be the
space of $k$-jets of $m$-dimensional submanifolds of $E$, i.e.~equivalence
classes of tangency of $m$-dimensional submanifolds up to order $k$. There is
a tower of fiber bundles
\[
E=J^{0}\longleftarrow J^{1}\longleftarrow\cdots\longleftarrow J^{k-1}%
\longleftarrow J^{k}\longleftarrow\cdots.
\]
If $S\subset E$ is an $m$-dimensional submanifold, its $k$-jet at the point
$e\in S$ is denoted by $[S]_{e}^{k}$. The $k$-jet prolongation of $S$ is the
submanifold $S^{(k)}\subset J^{k}$ defined as
\[
S^{(k)}:=\{[S]_{e}^{k}:e\in S\}.
\]
An $m$-dimensional contact element $R$ in $J^{k}$ of the form $R=T_{y}S^{(k)}%
$, $y\in S^{(k)}$, is called an $R$-\emph{plane}.\emph{ }Notice that the
$R$-plane $T_{y}S^{(k)}$, $y=[S]_{e}^{k}$, is completely determined by
$y^{\prime}:=[S]_{e}^{k+1}$. Accordingly, it is also denoted by $R_{y^{\prime
}}$. The correspondence $y^{\prime}\mapsto R_{y^{\prime}}$ is bijective.
There is a canonical distribution $C$ on $J^{k}$, the \emph{Cartan
distribution}. For $y\in J^{k}$, the \emph{Cartan plane }$C_{y}$ is spanned by
$R$-planes at $y$. It is easy to see, for instance using local coordinates
(see below), that $C$ is maximally non-integrable, i.e.~the characteristic
distribution is trivial, hence it is a canonical multicontact structure on
$J^{k}$. Moreover, there is a canonical isomorphism of vector bundles
$N=TJ^{k}/C\simeq V$, where, for $y\in J^{k}$,
$V_{y}:=T_{\underline{y}}J^{k-1}/R_{y}$, and $\underline y$ is the projection of $y$ down to $J^{k-1}$.

When $n=k=1$, $C$ is a contact distribution. In the following, I will assume,
for simplicity, $n>1$. The case $n=1$ is somewhat exceptional but can be
treated in a very similar way. Multicontact vector fields on $J^{k}$ are
characterized by the Lie-B\"{a}cklund theorem. Namely, let $\phi
:E\to E$ be a (local) diffeomorphism. There exists a unique
(local) diffeomorphism $\phi^{(k)}:J^{k}\to J^{k}$ preserving the
Cartan distribution, such that diagram
\[%
\xymatrix{ J^k \ar[r]^-{\phi^{(k)}}
\ar[d] & J^k \ar[d] \\
E \ar[r]^-{\phi} & E}%
\]
commutes. Diffeomorphism $\phi^{(k)}$ is called the $k$\emph{-th jet
prolongation} of $\phi$ and it is defined as follows. For $y=[S]_{e}^{k}\in
J^{k}$, $\phi(y):=[\phi(S)]_{\phi(e)}^{k}$. Infinitesimally, let $X$ be a
vector field on $E$. There exists a unique multicontact vector field $X^{(k)}$
on $J^{k}$ which projects on $X$. The flow of $X^{(k)}$ is the $k$-th jet
prolongation of the flow of $X$, and $X^{(k)}$ is called the $k$\emph{-th jet
prolongation of} $X$. Lie-B\"{a}cklund theorem states that (when $n>1$) every
multicontact field on $J^{k}$ is of the kind $X^{(k)}$. Summarizing, (when
$n>1$) there are Lie algebra isomorphisms $\mathfrak{X}(E)\simeq
\mathfrak{X}_{C}\simeq\Gamma_{\mathrm{Ham}}(N)$, $X\mapsto X^{(k)}%
\mapsto\theta(X^{(k)})$.

Let $y_{0}=[S_{0}]_{e_{0}}^{k}$ be a point in $J^{k}$. On $E$ choose, around
$e_{0}$, coordinates $(x^{i},u^{\alpha})$, $i=1,\ldots,m$, $\alpha = 1, \ldots, n$
adapted to $S_{0}$, i.e.~such that $S_{0}$ is locally given by $S_{0}%
:u^{\alpha}=f_{0}^{\alpha}(x)$. There is a neighborhood $U$
of $y$ in $J^{k}$ such that every point $y$ of $U$ is of the form
$y=[S]_{e}^{k}$ with $S$ locally given by $S:u^{\alpha}=f^{\alpha}%
(x)$. One can coordinatize $U$ by \emph{jet coordinates
}$(x^{i},u_{I}^{\alpha})$ defined as
\[
x^{i}(y)=x^{i}(e),\quad u_{I}^{\alpha}(y)=\frac{\partial^{|I|}f^{\alpha}%
}{\partial x^{I}}(x(e)),
\]
where $I=i_{1}\cdots i_{\ell}$ is a multindex ``
denoting'' multiple partial derivatives, i.e.~$\frac
{\partial^{|I|}}{\partial x^{I}}:=\frac{\partial^{\ell}}{\partial x^{i_{1}%
}\cdots\partial x^{i_{\ell}}}$, and $|I|{}:=\ell\leq k$ is the\emph{ lenght of
the multiindex}. The Cartan distribution is then locally spanned by vector
fields
\[
D_{i}:=\frac{\partial}{\partial x^{i}}+\sum_{|I|{}<k}u_{Ii}^{\alpha}%
\frac{\partial}{\partial u_{I}^{\alpha}},\quad\frac{\partial}{\partial
u_{K}^{\alpha}},\quad|K|{}=k,
\]
and its annihilator $C^{0}$ is locally spanned by \emph{Cartan forms}%
\[
\vartheta_{J}^{\alpha}:=du_{J}^{\alpha}-u_{Ji}^{\alpha}dx^{i},\quad|J|{}<k.
\]
Accordingly, the projection $\theta:\mathfrak{X}(J^{k})\to
\Gamma(N)$ and the curvature $R:\Gamma(C)\times\Gamma(C)\to
\Gamma(N)$ are locally given by
\[
\theta=\sum_{|J|{}<k}\vartheta_{J}^{\alpha}\otimes\theta\left(  \frac
{\partial}{\partial u_{J}^{\alpha}}\right)  ,\quad\text{and\quad}R=\sum
_{|J|{}=k-1}du_{Ji}^{\alpha}|_{C}\wedge dx^{i}|_{C}\otimes\theta\left(
\frac{\partial}{\partial u_{J}^{\alpha}}\right)  .
\]
If $X$ is a vector field on $M$ locally given by $X=X^{i}\frac{\partial
}{\partial x^{i}}+X^{\alpha}\frac{\partial}{\partial u^{\alpha}}$, the
corresponding multicontact vector field $X^{(k)}$ is given by
\begin{equation}
X^{(k)}=X^{i}D_{i}+\sum_{|I|{}\leq k}D_{I}\chi^{\alpha}\frac{\partial
}{\partial u_{I}^{\alpha}},\quad\chi^{\alpha}:=X^{\alpha}-u_{i}^{\alpha}X^{i},
\label{14}%
\end{equation}
where $D_{i_{1}\cdots i_{\ell}}:=D_{i_{1}}\circ\cdots\circ D_{i_{\ell}}$. The
Hamiltonian section $\theta(X^{(k)})$ is locally given by%
\[
\theta(X^{(k)})=\sum_{|J|{}<k}D_{J}\chi^{\alpha}\theta\left(  \frac{\partial
}{\partial u_{J}^{\alpha}}\right)  .
\]
The multisymplectic structure of the multisymplectization $(\widetilde
{M},\widetilde{\omega})$ of $(J^{k},C)$ is locally given by
\begin{equation}
\widetilde{\omega}=dp\wedge\Theta-\sum_{|J|{}=k-1}du_{Ji}^{\alpha}\wedge
dx^{i}\wedge\Theta_{\alpha}^{J}, \label{15}%
\end{equation}
where $\Theta_{\alpha}^{J}:=i_{\partial/\partial u_{I}^{\alpha}}\Theta$. Local
Formulas (\ref{14}) and (\ref{15}) allow one to find coordinate expressions
for the higher brackets in $\mathfrak{g}_{\bullet}(J^{k},C)$. Namely, let
$\mu_{1},\ldots,\mu_{\ell}$ be Hamiltonian forms in $K_{\mathrm{Ham}}$,
and let $X_{1},\ldots,X_{\ell}$ be associated Hamiltonian vector fields on
$J^{k}$, with $\theta(X_{s})=\sum D_{J}\chi_{s}^{\alpha}\theta(\partial
/\partial u_{J}^{\alpha})$, $s=1,\ldots,\ell$. A straightforward computation
shows that%
\begin{align*}
&  \lambda_{\ell}(\mu_{1},\ldots,\mu_{\ell})\\
&  =%
{\textstyle\sum_{|I_{1}|,\ldots,|I_{k}|<k}}
D_{I_{1}}\chi_{1}^{\alpha_{1}}\cdots D_{I_{\ell}}\chi_{\ell}^{\alpha_{\ell}%
}\left(  (-)^{\ell}p%
{\textstyle\sum_{|J|{}=k-1}}
du_{Ji}^{\alpha}\wedge dx^{i}\wedge\Theta_{\alpha_{1}\cdots\alpha_{\ell}%
\alpha}^{I_{1}\cdots I_{\ell}J}{}-dp\wedge\Theta_{\alpha_{1}\cdots\alpha
_{\ell}}^{I_{1}\cdots I_{\ell}}\right) \\
&  \quad\ \left.  +%
{\textstyle\sum_{s=1}^{\ell}}
(-)^{s}pD_{I_{1}}\chi_{1}^{\alpha_{1}}\cdots\widehat{D_{I_{s}}\chi_{s}%
^{\alpha_{s}}}\cdots D_{I_{\ell}}\chi_{\ell}^{\alpha_{\ell}}\left(
{\textstyle\sum_{|I|<k}}
(\tfrac{\partial}{\partial u_{I}^{\alpha}}D_{I}\chi_{s}^{\alpha}%
)\Theta_{\alpha_{1}\cdots\widehat{\alpha_{s}}\cdots\alpha_{\ell}}^{I_{1}%
\cdots\widehat{I_{s}}\cdots I_{\ell}}\right.  \right. \\
&  \quad\ \left.  +%
{\textstyle\sum_{|J|{}=k-1}}
\left(  X_{s}^{i}du_{Ji}^{\alpha}-D_{Ji}\chi_{s}^{\alpha}dx^{i}\right)
\wedge\Theta_{\alpha_{1}\cdots\widehat{\alpha_{s}}\cdots\alpha_{\ell}\alpha
}^{I_{1}\cdots\widehat{I_{s}}\cdots I_{\ell}J}\right)  ,
\end{align*}
where $\Theta_{\alpha_{1}\cdots\alpha_{\ell}}^{I_{1}\cdots I_{\ell}%
}:=i_{\partial/\partial u_{I_{1}}^{\alpha_{1}}}\cdots i_{\partial/\partial
u_{I_{\ell}}^{\alpha_{\ell}}}\Theta$.

\quad
\end{document}